\documentclass[11pt]{amsart}
\usepackage{amsmath,amssymb,verbatim,eucal,amscd,color,mathrsfs,appendix}
\usepackage[all]{xy}
\usepackage{amsfonts,latexsym,bm}
\usepackage{tikz-cd}

\numberwithin{equation}{section}
\usepackage{iftex}
\ifPDFTeX
  \usepackage[T1]{fontenc}
  \usepackage[utf8]{inputenc}
  \usepackage{textcomp} % provide euro and other symbols
\else % if luatex or xetex
  \usepackage{unicode-math}
  \defaultfontfeatures{Scale=MatchLowercase}
  \defaultfontfeatures[\rmfamily]{Ligatures=TeX,Scale=1}
\fi
\IfFileExists{upquote.sty}{\usepackage{upquote}}{}
\IfFileExists{microtype.sty}{% use microtype if available
  \usepackage[]{microtype}
  \UseMicrotypeSet[protrusion]{basicmath} % disable protrusion for tt fonts
}{}
\makeatletter
\@ifundefined{KOMAClassName}{% if non-KOMA class
  \IfFileExists{parskip.sty}{%
    \usepackage{parskip}
  }{% else
    \setlength{\parindent}{0pt}
    \setlength{\parskip}{6pt plus 2pt minus 1pt}}
}{% if KOMA class
  \KOMAoptions{parskip=half}}
\makeatother
\usepackage{xcolor}
\ifLuaTeX
  \usepackage{selnolig}  % disable illegal ligatures
\fi
\IfFileExists{bookmark.sty}{\usepackage{bookmark}}{\usepackage{hyperref}}
\IfFileExists{xurl.sty}{\usepackage{xurl}}{} % add URL line breaks if available
\hypersetup{
  colorlinks=true,
  linkcolor=red,
  citecolor=blue,
  urlcolor=magenta,
  pdfcreator={LaTeX via pandoc}}

\theoremstyle{plain}

\newtheorem{theorem}{Theorem}[section]				
\newtheorem{proposition}[theorem]{Proposition}		
\newtheorem{corollary}[theorem]{Corollary}
\newtheorem{lemma}[theorem]{Lemma}

\theoremstyle{definition}

\newtheorem{definition}[theorem]{Definition}
\newtheorem{remark}[theorem]{Remark}
\newtheorem{question}[theorem]{Question}
\newtheorem{example}[theorem]{Example}

\newcommand{\ABbb}{\mathbb A}

\newcommand{\CBbb}{\mathbb C}

\newcommand{\PBbb}{\mathbb P}
\newcommand{\QBbb}{\mathbb Q}

\newcommand{\ZBbb}{\mathbb Z}

\newcommand{\Acal}{\mathcal A}

\newcommand{\Ecal}{\mathcal E}

\newcommand{\Hcal}{\mathcal H}

\newcommand{\Ocal}{\mathcal O}

\newcommand{\Gfrak}{\mathfrak G}

\DeclareMathOperator{\Ext}{Ext}

\DeclareMathOperator{\Gm}{\mathbb{G}_{m}}
\DeclareMathOperator{\Spec}{Spec}

\newcommand{\Pic}{{\rm Pic}}

\AtBeginDocument{%
	\def\MR#1{}
}
\AtBeginDocument{%
	\def\MR#1{}
}

\begin{document}

\author{Arnab Roy}
\title{On smooth affine surfaces with the cohomology of a smooth curve}
%\title{On Jouanolou's trick}  
\address{School of Mathematics, Tata Institute of Fundamental Research, Homi Bhabha Road, NavyNagar, Colaba, Mumbai 400005, India}
\email{aroy@math.tifr.res.in}

\begin{abstract}
We prove that if a smooth affine complex surface has the same integral mixed Hodge structure as a smooth curve whose compactification is of positive genus, then it admits a smooth morphism to the curve with fibers isomorphic to $\ABbb^1$. In particular, every such surface is the total space of a torsor for a line bundle on the curve.  The motivation comes from Jouanolou's trick, to which this result gives a kind of converse in dimension $2$. We also give an example in the case of $\PBbb^1$ which has the same integral mixed Hodge structure as $\PBbb^1$ and simply connected but does not admit a smooth $\ABbb^1$ fibration over $\PBbb^1$.

\end{abstract}
\maketitle
\section{Introduction}

We prove the following:
\begin{theorem}\label{thetheorem}
Let $A$ be a smooth affine surface over $\CBbb$ such that there is an isomorphism of $\QBbb$-mixed Hodge structures (as a ring)
\begin{equation}\label{basic}
H^*(A,\QBbb)\cong H^*(C,\QBbb),
\end{equation}
where $C$ is a smooth curve with compactification $\bar C$ of genus $g\geq1$ and satisfies one of the following:
\begin{enumerate}
    \item $H^*(A,\ZBbb)\cong H^*(C,\ZBbb)$ as cohomology rings,
    \item $\pi_1(A)$ is torsion free.
\end{enumerate} Then $A$ admits a smooth fibration over $C$ whose fibers are isomorphic to $\ABbb^1$.
\end{theorem}

%It follows that various motivic invariants of $A$ and $C$ agree. 

The motivation for this result comes from the well-known Jouanolou's trick (Lemma~1.5 of \cite{Jouanolou}).  Recall that for a quasiprojective variety $X$ over a field $k$, Jouanolou constructs an affine variety $A$ and a surjective morphism $A \to X$ which is a torsor for a vector bundle on $X$.  It follows that the motivic invariants of $A$ and $X$ agree.  In particular, when $k=\CBbb$, $\pi_1(A)\cong\pi_1(X)$ and pullback gives an isomorphism of integral mixed Hodge structures
\begin{equation}\label{jouanolou:}
H^*(A,\ZBbb)\cong H^*(X,\ZBbb). 
 \end{equation}
Theorem~\ref{thetheorem} is a kind of converse, in low dimensions:  If $A$ and $X$ have the same integral mixed Hodge structure (as rings) or same rational mixed Hodge structure (as rings), $\pi_1(A)$ is torsion free, smooth compactification of $X$ is of positive genus, and $\dim A= 2\dim X=2$, then $A$ is a smooth affine space fibration over $X$. Note that we don't need isomorphism in $\pi_1$. In theorem~\ref{thm2} we show that it is possible to have the same result if one assumes isomorphism in rational mixed Hodge structures (as graded vector spaces) and in $\pi_1$ of $A$ and $X$. Example~\ref{example} shows that these conditions are necessary. We also remark that using \cite[Thm. 11]{Dubouloz} one can show that the ring isomorphism of rational mixed Hodge structures as in theorem~\ref{thetheorem} implies that $A$ is an \'etale line bundle torsor over a smooth (not necessarily locally separated) algebraic space which surjects to $C$. We would like to mention that in \cite{MR1361739} the authors have classified smooth open surfaces with non-negative Eular characteristic.

%If $X$ is projective, then necessarily $\dim A \geq 2\dim X$ and in fact, the torsor in Jouanolou's trick can always be arranged to satisfy $\dim A=2\dim X$ (cf. the discussion in Section \ref{preliminaries}).  Thus Theorem \ref{thetheorem} is a sort of converse to Jouanolou's trick in dimension 2.    

The condition $\dim A = 2\dim X$ is certainly necessary:  If $X$ is a point and $\dim A=2$ one cannot conclude that $A$ is $\ABbb^2$, cf. C.~P.~Ramanujam's famous example \cite{Ramanujam} of a contractible affine surface which is not isomorphic to the affine plane. It is worthwhile to note that while $\ABbb^1$ has dimension 1 still the same example works for the case where $X\cong\ABbb^1$. 
Note also that an  isomorphism of singular cohomology as in (\ref{basic})(even with isomorphism in $\pi_1)$ is not enough to conclude the existence of a fibration $A \to C$.  For example, $A=\Gm\times\Gm$ has the same homotopy type as an elliptic curve, but admits no map to an elliptic curve. Of course, there is no isomorphism of mixed Hodge structures as in (\ref{basic}). 

We give a counter-example~\ref{example2} in the case $H^*(A,\ZBbb)\cong H^*(\PBbb^1,\ZBbb)$ as in (\ref{basic}). The affine surface in our example is simply connected and has fundamental group at infinity isomorphic to a finite cyclic group. We also show that one needs stricter restriction on the order of the fundamental group at infinity (cf.remark~\ref{fundgpinfty}) in order to get a smooth $\ABbb^1$ fibration over $\PBbb^1$. It is clear from \cite{Ramanujam} that in case the curve is $\ABbb^1$ assuming simply-connected at infinity gives us our required result. In case of other affine rational curves $C$ we do not know what can be said when $H^*(A,\ZBbb) \cong H^*(C,\ZBbb)$ as in (\ref{basic}) . One could also ask for analogues of the theorem over other fields.

Here is a brief sketch of the proof of Theorem \ref{thetheorem}: Take a smooth projective compactification $S$ of $A$ with S.N.C. boundary $\partial$. Arguments using mixed Hodge theory show that $\partial$ is a tree of curves with exactly one component $\partial'$ of positive genus, say $g$, and that the remaining components are all rational, and that $S$ has no global holomorphic $2$-forms. We show that the Albanese image of $S$ is a smooth curve $X$ of genus $g$ and the Albanese morphism  restricts to an isomorphism of $\partial'$ with $X$.  We then prove that the general fiber of the Albanese map is $\PBbb^1$. Passing to the minimal model of $S$, which is necessarily a projective bundle over $X$, we show that $\partial'$ is a section of the projective bundle. Then we prove that the restriction of the map $S\to X$ to $A$, has fibers with reduced structure isomorphic to $\ABbb^1$ with finitely many unreduced fibers. Moreover we show that, when the map $A\to X$ is smooth then $S$ is itself minimal. Then using \cite{Xiao} or unimodularity of the cup product pairing we show that, under the conditions of theorem~\ref{thetheorem} the morphism $A\to X$ is smooth. Hence $A$ admits an $\ABbb^1$-fibration over its image $C$ in $X$.

Beyond surfaces, the analogous question for affine varieties with the same invariants as an abelian variety also appears to be interesting:

\begin{question}
Let $V$ be a smooth affine variety of dimension $2n$ ($n>1$) over $\CBbb$. Suppose that there is an isomorphism of mixed Hodge structures
\[
H^*(V,\ZBbb)\cong H^*(\Acal,\ZBbb),
\]
where $\Acal$ is an abelian variety of dimension $n$. Is there a geometric relation between $V$ and $\Acal$? 
\end{question}

Variants of the question are possible, e.g. one could further make assumptions on $\pi_1$, or on motivic invariants of $V$ and $\Acal$. The answers are not clear to us even when $n=2$.

\medskip

{\bf Acknowledgements.}
I am deeply grateful to Prof. Arvind Nair for suggesting the use of the decomposition theorem in the proof of Theorem \ref{mainthm}, for many helpful discussions, and for his constant support. I thank Prof. Najmuddin Fakhruddin for his encouragement and valuable comments on this work. I am also grateful to Mr. Aniket Chakraborty for introducing me to Jouanolou's trick, and to Mr. Swapnajit Das and Mr. Manish Kumar for several helpful discussions.

\medskip

\section{Preliminaries}
\label{preliminaries}
The following notation and definitions are used throughout this article. 

{\bf Notation.} For a variety $X/\CBbb$ and a $\ZBbb$-module $R$, we denote by $H^*(X,R)$ (respectively, $H^*_c(X,R)$) the singular cohomology (respectively, the singular cohomology with compact supports) of the associated complex analytic space with coefficients in $R$.
\begin{definition}[Dual graph of a curve]\label{dualgraph}
Let $X$ be a proper one-dimensional scheme of finite type  over a field $k$ with at most nodal singularities. The {\bf dual graph} of $X$ is the one-dimensional CW complex defined as follows:  The vertices correspond to the irreducible components of $X$ and for each node there is an edge joining the vertices corresponding to the two irreducible components which share the node. Clearly $X$ and $X^{\mathrm{red}}$ have the same dual graph. 
    
\end{definition}
 A subgraph is a subcomplex of the original graph and it is called \textbf{induced} if every edge between vertices of the subgraph in the original graph appears in the subgraph. 
 \begin{definition}
     Let $D$ be a $S.N.C.$ curve in a smooth surface $S$. An irreducible component of $D$ is called a \textbf{leaf} if it intersects the rest of $D$ at a single point. An irreducible component is called a \textbf{bridge} if intersects the rest of $D$ at $2$ points.
 \end{definition}
Now we briefly recall Jouanolou's construction. One first constructs a vector bundle torsor $V_n$ over $\PBbb^n$ which is an affine variety.   It is clear from the explicit construction of the torsor given in \cite{Jouanolou} that the fibers of this torsor are isomorphic to $\ABbb^n$. Thus the total space of the torsor has twice the dimension of the base. Next, for any projective variety of dimension $n$, one can take a finite map to $\PBbb^n$ and pull back $V_n$ along this map. Since finite maps are affine, the resulting pullback is affine and has dimension $2n$. For a quasi-projective variety, one can take a projective compactification whose boundary is an effective Cartier divisor. Pulling back the torsor over the compactification to the quasi-projective variety then gives the required construction. Thus, in all cases, one can arrange a torsor in which the total space has twice the dimension of the base.

\section{Smooth compactification and boundary}

 We start by introducing the main objects of our study.
\begin{definition}[Condition $J$]\label{J}
A pair $A \subset S$ consisting of a smooth affine surface over $\CBbb$ and a smooth projective compactification of it is said to satisfy condition $J$ if 
\begin{enumerate}
    \item $H^*(A,\QBbb)$ is the rational mixed Hodge structure (as a ring) of a smooth curve, denote $2g=\dim (gr^W_1(H^1(A,\QBbb)))$ and 
    \item $\partial:=S-A$ is $S.N.C.$ and contains no $-1$-curves on $S$ which are leaves in $\partial$. In case $\partial$ is a tree then it also contains no $-1$ curves on $S$ which are bridges in $\partial$.
\end{enumerate}
\end{definition}

Note that given an affine surface $A$ there always exists a smooth projective compactification $S$ of $A$, satisfying point (2) of the above condition. The non-negative integer $g$ in condition $(1)$ is the genus of the smooth compactification of the curve (on which the rational M.H.S. of $A$ is modeled). Now a smooth curve can either be affine or proper. So we break the pairs $(A,S)$ satisfying condition $J$ into two cases:
\begin{enumerate}
    \item When $H^*(A,\QBbb)$ is modeled on the M.H.S. of a smooth affine curve, we call it the \textit{affine case},
    \item when $H^*(A,\QBbb)$ is modeled on the M.H.S. of a smooth proper curve, we call it the \textit{proper case}.
\end{enumerate}
Note that the condition on the ring structure of $H^*(A,\QBbb)$ is vacuous in the affine case as $H^2(A,\QBbb)=0$.

In the next couple of lemmas, we describe the structure of $\partial$.

\begin{lemma}\label{lem1}
Under condition $J$ in Definition \ref{J}, the dual graph of $\partial$ is a connected tree. If the $g_i$'s denote the arithmetic genera of the irreducible components of $\partial$, then
\[
\sum g_i=g.
\]

\begin{proof}
First note that, by condition $J$, $H^1(A,\QBbb)$ has weights $1,2$ and $H^2(A,\QBbb)$ is pure of $(1,1)$ type. Thus, by Poincar\'e duality, $H^3_c(A,\QBbb)$ has weights $2,3$, $H^2_c(A,\QBbb)$ is pure of $(1,1)$ type, and
$
H^i_c(A,\QBbb)=0
$
for $i\leq1$.
Now consider the long exact sequence of compactly supported cohomology for the pair $(S,\partial)$:
\begin{equation}
0\to H^0(S,\QBbb)\to H^0(\partial,\QBbb)\xrightarrow{f} H^1_{c}(A,\QBbb)=0\to H^1(S,\QBbb)\to H^1(\partial,\QBbb)\to...
\end{equation}

\begin{equation}
H^1(\partial,\QBbb)\xrightarrow{g} H^2_{c}(A,\QBbb)\to H^2(S,\QBbb)\to H^2(\partial,\QBbb)\xrightarrow{h} H^3_c(A,\QBbb)\to H^3(S,\QBbb)\to H^3(\partial,\QBbb)=0
\end{equation}

\begin{equation}
0\to H^4_c(A,\QBbb)\to H^4(S,\QBbb)\to H^4(\partial,\QBbb)=0.
\end{equation}

Since $\partial$ is a reduced $S.N.C.$ divisor, $H^1(\partial,\QBbb)$ has weights at most $1$, while $H^2(\partial,\QBbb)$ is pure of weight $2$. As the above sequences are exact sequences of $M.H.S.$, the maps $f$, $g$ must vanish. Also as $H^3(S,\QBbb)$ is pure of weight 3, $H^3(S,\QBbb)\cong gr^W_3(H^3_c(A,\QBbb)$ and $H^2(\partial,\QBbb)\twoheadrightarrow gr^W_2(H^3_c(A,\QBbb)$.

The first consequence is that
$
H^0(\partial,\QBbb)=H^0(S,\QBbb)=\QBbb
$
so $\partial$ is connected. 
Moreover, since $S$ is smooth and projective, $H^1(S,\QBbb)$ is pure of weight $1$. As
$
H^1(S,\QBbb)\cong H^1(\partial,\QBbb),
$
it follows that $H^1(\partial,\QBbb)$ is also pure of weight $1$.
We also have
$
gr^W_3(H^3_c(A,\QBbb))\cong H^3(S,\QBbb),
$
and hence, by Poincar\'e duality,
\[
H^1(S,\QBbb)\cong gr^W_1(H^1(A,\QBbb))\cong \QBbb^{\oplus 2g}.
\]

Another interesting consequence of the above exact sequences is that $H^2(S,\QBbb)$ is an extension of $H^2_c(A,\QBbb)$ and $\ker(h)\subset H^2(\partial,\QBbb)$, both of which are pure of weight $2$ and of Hodge type $(1,1)$. Consequently,
$
H^2(S,\CBbb)
$
is spanned by $(1,1)$-classes. Equivalently,
\[
H^2(S,\Ocal_S)=H^0(S,\Omega^2_S)=0.
\]
As $\partial$ is an $S.N.C.$ divisor, we have
\[
gr^W_1\bigl(H^1(\partial,\QBbb)\bigr)\cong
\QBbb^{\oplus 2\sum g_i},
\qquad
gr^W_0\bigl(H^1(\partial,\QBbb)\bigr)
\cong
\QBbb^{\oplus (N-l+1)},
\]
where the $g_i$'s are the arithmetic genera of the irreducible components of $\partial$, $N$ is the number of nodes, and $l$ is the number of irreducible components of $\partial$.

The previous computation shows that $H^1(\partial,\QBbb)$ is pure of weight $1$. Hence,
$
N-l+1=0,
$
which immediately implies that the dual graph of $\partial$ is a tree. Moreover, since
\[
\dim H^1(\partial,\QBbb)=\dim H^1(S,\QBbb),
\]
we obtain
$
\sum_{i=1}^{l}g_i=g.
$ One also notices that \[l=\dim(H^2(\partial,\QBbb))=\dim(H^2(S,\QBbb))-\dim(H^2(A,\QBbb))+\dim(gr^W_2(H^1(A,\QBbb)).\]
\end{proof}
\end{lemma}

\begin{remark}
An immediate consequence of the above lemma in the case $g=0$ is that $g_i=0$ for all $i$. Thus, $\partial$ is a tree of $\PBbb^1$'s.
\end{remark}

Now we focus on the case $g\geq1$. Consider the Albanese morphism
\[
\alpha:S\to Alb(S).
\]
By the universal property of the Albanese morphism, $Alb(S)$ is generated by the image $\alpha(S)$. Moreover,
$
\dim Alb(S)=\dim H^0(S,\Omega^1_S)=g\geq1,
$
so $\alpha(S)$ is positive-dimensional.

The fact that $S$ has no global $2$-forms implies that $\alpha(S)$ is one-dimensional. Indeed, suppose that
$
\dim\alpha(S)=2.
$
Then $\alpha$ is generically finite onto its image. Since we are working in characteristic $0$, $\alpha$ is generically étale over its image. Choose a smooth point
$
p\in\alpha(S)
$
such that $\alpha$ is étale over a neighbourhood of $p$ in $\alpha(S)$.
Take two linearly independent local $1$-forms $\omega_1,\omega_2$ on $\alpha(S)$ near $p$. Then
$
\omega_1\wedge\omega_2
$
is a non-zero local $2$-form at $p$. Since
\[
\alpha(S)\hookrightarrow Alb(S)
\]
is an immersion, these local $1$-forms lift to local $1$-forms
$
\omega_1',\omega_2'
$
on $Alb(S)$ near $p$. As the cotangent bundle of $Alb(S)$ is trivial, the local 2-form
$
\omega_1'\wedge\omega_2'
$
extends to a non-zero global $2$-form on $Alb(S)$.

Pulling this form back to $S$ via $\alpha$ gives a global $2$-form on $S$. This pull-back is non-zero because, over a preimage of $p$, it coincides with the pull-back of the non-zero local form
$
\omega_1\wedge\omega_2
$
and $\alpha$ is \'etale near $p$. This contradicts the fact that
$
H^0(S,\Omega_S^2)=0.
$

Let $X$ be the image of $S$ in $\operatorname{Alb}(S)$. We claim that $X$ is smooth and that $p:S\to X$ has connected fibers. Taking the Stein factorization of $p$, we have
\[
\begin{tikzcd}
	S & {\hat{X}} & X & {\operatorname{Alb}(S)}
	\arrow["h"', from=1-1, to=1-2]
	\arrow["p"', bend right=40, from=1-1, to=1-3]
	\arrow["\alpha", bend left=40, from=1-1, to=1-4]
	\arrow["f"', from=1-2, to=1-3]
	\arrow["i"', hook, from=1-3, to=1-4].
\end{tikzcd}
\]
where $\hat{X}$ is normal (as $S$ is normal), $h$ has connected fibers, and $f$ is finite. The map $h$ induces a surjection $\alpha_h:\operatorname{Alb}(S)\to\operatorname{Alb}(\hat{X})$, while $i\circ f$ induces $\alpha_f:\operatorname{Alb}(\hat{X})\to\operatorname{Alb}(S)$. Note that $\alpha=(\alpha_f\circ\alpha_h)\circ\alpha$, then by the universal property of the Albanese morphism, $\alpha_f\circ\alpha_h$ is an isomorphism. Surjectivity of $\alpha_h$ then implies that $\alpha_f$ is an isomorphism. Since the Albanese morphism of the smooth projective curve $\hat{X}$ is an embedding, it follows that $f$ is an isomorphism. Thus $X\cong\hat{X}$ is smooth of genus $g=\dim \text{Alb}(S)$, and $p$ has connected fibers. 
\begin{remark}
    The above argument uses the fact that base field is characteristic 0. In the proof of the fact that the image of the Albanese morphism is a curve we use that in case of characteristic zero a finite map is generically \'etale.   
\end{remark}

We obtain the following corollary.

\begin{corollary}\label{fundamentalgroup}
In the case $g\geq1$,
\begin{itemize}
\item $\partial$ has a unique irreducible component $\partial'$ of genus $g$, and all the other irreducible components are $\PBbb^1$'s. Moreover, the restriction of the map $p$ to $\partial'$ induces an isomorphism
\[
\partial'\xrightarrow{\sim} X.
\]

\item The maps 
$
\partial'\hookrightarrow S
$
and 
$
p:S\to X
$
induce isomorphisms between the corresponding fundamental groups.
\end{itemize}

\begin{proof}
First note that the restriction
\[
p|_{\partial}:\partial\to X
\]
is surjective. Otherwise, since $\partial$ is connected, its image in $X$ would be a point. Consequently, $A$ would contain a fiber of the map $p$, which is impossible because $p$ is proper while $A$ is affine.

Hence at least one irreducible component of $\partial$ surjects onto $X$. Therefore, at least one of the $g_i$ satisfies
$
g_i\geq g.
$
On the other hand,
\[
\sum_{i=1}^{l}g_i=g,
\]
so the only possibility is that exactly one irreducible component has genus $g$, while all the remaining components are $\PBbb^1$'s. Moreover, this unique component of genus $g$ surjects onto $C$. We denote it by $\partial'$.

By \cite[Th. 2]{Goodman}, $\partial$ supports an ample divisor. Hence, by the Lefschetz hyperplane theorem (cf. \cite{Morse}),
$
\pi_1(\partial)\rightarrow\pi_1(S)
$
is surjective.
Since $\partial$ is an $S.N.C.$ divisor and every irreducible component other than $\partial'$ is isomorphic to $\PBbb^1$, the inclusion
$
\partial'\hookrightarrow\partial
$
induces an isomorphism
$
\pi_1(\partial')\xrightarrow{\sim}\pi_1(\partial).
$
Consequently,
$
\pi_1(\partial')\longrightarrow\pi_1(S)
$
is also surjective.
Similarly, the Lefschetz hyperplane theorem implies that
$
H_1(\partial',\ZBbb)\longrightarrow H_1(S,\ZBbb)
$
is surjective. Since $H_1(\partial',\ZBbb)$ is free of rank $2g$, while
\[
H^1(S,\QBbb)\cong H_1(S,\ZBbb)\otimes\QBbb
\]
also has rank $2g$, $H_1(\partial',\ZBbb)\to H_1(S,\ZBbb)$ is an isomorphism. In particular, $H_1(S,\ZBbb)$ is free.

It follows that the composite morphism
\[
\partial'\hookrightarrow S\xrightarrow{\alpha}Alb(S)
\]
is precisely the Albanese morphism of $\partial'$. Since the map
$
\partial'\longrightarrow Alb(S)
$
is an embedding, the induced morphism
$
p|_{\partial'}:\partial'\longrightarrow X
$
is an isomorphism. This proves the first assertion.

The second assertion follows immediately from the commutative diagram of fundamental groups
\[
\begin{tikzcd}
	{\pi_1(\partial')} && {\pi_1(S)} \\
	& {\pi_1(X)}
	\arrow[two heads, from=1-1, to=1-3]
	\arrow["\cong"', from=1-1, to=2-2]
	\arrow["{p_*}", from=1-3, to=2-2]
\end{tikzcd}
\]
which completes the proof.
\end{proof}
\end{corollary}

\begin{remark}
Although we do not obtain a morphism to a curve in the case $g=0$, all the irreducible components of $\partial$ are $\PBbb^1$'s. Hence, by the Lefschetz hyperplane theorem, $\pi_1(S)$ is trivial.
\end{remark}
\begin{remark}\label{surjection}
    Note that in the proper case the map $A\to X$ is a surjection. Clearly as the pullback maps $H^1(S,\QBbb)\to H^1(A,\QBbb)$ and $H^1(X,\QBbb)\to H^1(S,\QBbb)$ are isomorphisms so is the map $H^1(X,\QBbb)\to H^1(A,\QBbb)$. Now by condition $J$ we know that the cup-product on $H^1(A,\QBbb)$ is non-degenerate. Thus naturality of cup product pairing implies $H^2(X,\QBbb)\to H^2(A,\QBbb)$ is non-zero, as a result $A\to X$ cannot factor through a proper subset of $X$.
\end{remark}

\section{Study of the fibration and the main theorem}

In this section, we study the fibration induced by the Albanese morphism
\[
S\to X.
\]
 Our goal is to prove that the general fiber of this morphism is $\PBbb^1$. Since $S$ is a smooth surface, the general fiber is smooth. Thus, it is enough to determine the genus of the general fiber, which is captured by its first singular cohomology (with $\QBbb$-coefficients).

Before doing so, we prove a general lemma concerning a fibration over a curve with a unique singular fiber.

\begin{theorem}\label{mainthm}
Let $X$ be a smooth variety of dimension $d$, and let
\[
f:X\to C
\]
be a flat projective morphism with connected fibers onto a smooth curve $C$. Suppose that $f$ has exactly one singular fiber, lying over a closed point $c\in C$. If $F$ denotes a general fiber and $F_c$ denotes the fiber over $c$, then
\[
\dim H^1(F,\QBbb)\geq\dim H^1(F_c,\QBbb).
\]

\begin{proof}
Let
\[
j:C^*=C\setminus\{c\}\hookrightarrow C
\]
denote the open immersion.

Consider the shifted constant sheaf $
\underline{\QBbb}_{X}[d]
$ on $X$, which is the perverse sheaf underlying a pure Hodge module of weight $d$ (in the sense of Saito cf. \cite{Saito1},\cite{Saito2},\cite{BBD}).  We denote perverse cohomology sheaves of a complex $K^\bullet$ by
\[
{}^p\Hcal^i(Rf_*(K^\bullet)).
\]

Since $f$ is projective, the decomposition theorem (cf. \cite[(4.5.4)]{Saito2} or \cite[Th.~6.2.5]{BBD}) implies that
$
Rf_*(\underline{\QBbb}_X[d])
$
is noncanonically isomorphic to the shifted direct sum of its perverse cohomology sheaves. Moreover,
$
{}^p\Hcal^i(Rf_*(\underline{\QBbb}_X[d]))
$
is pure, and each perverse cohomology sheaf is canonically isomorphic to a direct sum of intermediate extensions.

Since $f$ is smooth over $C^*$, the decomposition theorem gives
\begin{align}
Rf_*(\underline{\QBbb}_X[d])|_{C^*}
&=
\bigoplus_{k=0}^{2(d-1)}
\left(
R^kf_*(\underline{\QBbb}_X)|_{C^*}[1]
\right)
[-k-1+d]
\\
&=
\bigoplus_{i=-(d-1)}^{(d-1)}
\left(
R^{i+d-1}f_*(\underline{\QBbb}_X)|_{C^*}[1]
\right)
[-i].
\end{align}

Let
\[
L^i:=R^{i+d-1}f_*(\underline{\QBbb}_X)|_{C^*},
\qquad -(d-1)\leq i\leq(d-1).
\]
Then
$
j^*\left(
{}^p\Hcal^i(Rf_*(\underline{\QBbb}_X[d]))
\right)
=
L^i[1].
$ Also note that $L^i$'s are local systems on $C^*$.

Consequently,
\[
{}^p\Hcal^i(Rf_*(\underline{\QBbb}_X[d]))
=
j_{!*}(L^i[1])\oplus S_i,
\qquad
-(d-1)\leq i\leq(d-1),
\]
where $j_{!*}$ denotes the intermediate extension and $S_i$ is a skyscraper sheaf supported at $c$.

Since $C$ is a smooth curve, we have
$
j_{!*}=j_*.
$
Therefore,
\[
{}^p\Hcal^i(Rf_*(\underline{\QBbb}_X[d]))
=
j_*L^i[1]\oplus S_i,
\qquad
-(d-1)\leq i\leq(d-1).
\]

All the remaining perverse cohomology sheaves are skyscraper sheaves supported at $c$. We continue to denote them by
\[
{}^p\Hcal^i(Rf_*(\underline{\QBbb}_X[d]))=S_i,
\qquad
i>(d-1)\ \text{or}\ i<-(d-1).
\]

By the Hard Lefschetz theorem for perverse cohomology sheaves,
\[
{}^p\Hcal^{-i}(Rf_*(\underline{\QBbb}_X[d]))
\cong
{}^p\Hcal^i(Rf_*(\underline{\QBbb}_X[d])).
\]

Hence
$
S_{-i}\cong S_i,\; L^{-i}\cong L^i
$
for all $i$. 
Taking the stalk of
$
Rf_*(\underline{\QBbb}_X)
$
at $c$, we obtain
\[
Rf_*(\underline{\QBbb}_X)|_c
=
\bigoplus_{i=-(d-1)}^{d-1} j_*L^{i}|_c[1-i-d]
\oplus
\bigoplus_{t\in\ZBbb}S_t[-t-d].
\]

Since $f$ is flat, $F_c$ has no cohomology in degrees $\geq2 d-1$. We conclude that
$
S_{t}=0
$
for all \[t+d\geq 2d-1\implies t\geq d-1.\] By Hard Lefschetz,
$
S_{t}=0
$
for all $t\leq-(d-1)$ as well.

Therefore,
\[
H^1(F_c,\QBbb)
= H^1(Rf_*(\underline{\QBbb}_X)|_c)=
j_*L^{-d+2}|_c
=
j_*\!\left(
R^1f_*(\underline{\QBbb}_X)|_{C^*}
\right)\Big|_c.
\]

Now
$
R^1f_*(\underline{\QBbb}_X)|_{C^*}
$
is a local system of rank
$
\dim H^1(F,\QBbb).
$
Hence the stalk of
$
j_*\!\left(
R^1f_*(\underline{\QBbb}_X)|_{C^*}
\right)
$
at $c$ has rank at most
$
\dim H^1(F,\QBbb),
$ (in fact it is same as the local monodromy invariance in $H^1(F,\QBbb)$.  
This proves the theorem.
\end{proof}
\end{theorem}

\begin{remark}\label{rem 2}
\begin{itemize}
\item[(1)] The above computation implies that if
\[
\dim H^1(F_c,\QBbb)=\dim H^1(F,\QBbb),
\]
then the local monodromy around $c$ for the local system
$
R^1f_*(\underline{\QBbb}_X)|_{C^*}
$
is trivial. In other words,
$
j_*\bigl(R^1f_*(\underline{\QBbb}_X)|_{C^*}\bigr)
$
is a local system at $c$.

\item[(2)] Let's focus on the case $d=2$. Then note that $L^{-1}$ is the constant local system, since its stalks compute the $H^0$ of the fibers. Since $L^{-1}\cong L^1$, we obtain
\[
j_*(L^{-1})=j_*(L^1)=\underline{\QBbb}_C.
\]
Moreover, we have the following description of $Rf_*(\underline{\QBbb}_X)$:
\[
Rf_*(\underline{\QBbb}_X)
=
(\underline{\QBbb}_C\oplus S_0)[-2]
\oplus
j_*(L^0)[-1]
\oplus
\underline{\QBbb}_C[0].
\]
\end{itemize}

Taking the stalk at $c$, we see that the rank of $S_0$ is $h-1$, where $h$ is the number of irreducible components of $F_c$. Moreover, taking cohomology sheaves, we obtain
\[
R^0f_*(\underline{\QBbb}_X)=\underline{\QBbb}_C,\qquad
R^1f_*(\underline{\QBbb}_X)=j_*\bigl(R^1f_*(\underline{\QBbb}_X)|_{C^*}\bigr),\qquad
R^2f_*(\underline{\QBbb}_X)=\underline{\QBbb}_C\oplus S_0.
\]
Combining this with part (1), we conclude that whenever
\[
\dim H^1(F_c,\QBbb)=\dim H^1(F,\QBbb),
\]
the sheaf $R^1f_*(\underline{\QBbb}_X)$ is locally constant at $c$.
\end{remark}

We are now ready to prove that the general fiber of the fibration
$
S\to C
$
is $\PBbb^1$.

\begin{proposition}
Let $(S,A)$ be a pair satisfying Condition $J$ in definition \ref{J}. Then the general fiber of the morphism
$
p:S\to X
$
is $\PBbb^1$.

\begin{proof}
Under Condition $J$, we know that
 $\partial=S\setminus A$ is a connected $S.N.C.$ divisor in $S$ whose dual graph is a tree, with one irreducible component isomorphic to $X$ and all the remaining components isomorphic to $\PBbb^1$.

We first compute the cohomology of $\partial$. To this end, we note the following elementary fact. Let $T$ be a tree of curves, and suppose that we attach a leaf $L\cong\PBbb^1$ to it. Then
\[
H^0(T\cup L,\QBbb)=\QBbb,\qquad
H^1(T\cup L,\QBbb)=H^1(T,\QBbb),
\]
and
\[
H^2(T\cup L,\QBbb)=H^2(T,\QBbb)\oplus\QBbb.
\]
\iffalse
Indeed, write
\[
T\cup L=(T-(T\cap L))\sqcup(L-(T\cap L))\sqcup\{p\},
\]
where $p=T\cap L$. We have
\[
H^0_c(T-p)=0,\qquad
H^1_c(T-p)=H^1(T,\QBbb),\qquad
H^2_c(T-p)=H^2(T,\QBbb),
\]
and similarly
\[
H^0_c(L-p)=H^1_c(L-p)=0,\qquad
H^2_c(L-p)=\QBbb.
\]
The desired conclusion follows immediately from the long exact sequences for compactly supported cohomology and ordinary cohomology.
\fi
Applying this observation inductively, we obtain
$
H^0(\partial,\QBbb)=\QBbb,
H^1(\partial,\QBbb)=H^1(X,\QBbb),
$
and
$
H^2(\partial,\QBbb)=\QBbb^{\oplus(r+1)},
$
where $r$ denotes the number of $\PBbb^1$-components of $\partial$.

Now consider the open-closed decomposition of $S$ into $A$ and $\partial$. We have
\[
\chi(S)=\chi(H^*_c(A,\QBbb))+\chi(\partial).
\]
By Poincar\'e duality,
\[
\chi(H^*_c(A,\QBbb))=\chi(A),
\]
while the above computation gives
$
\chi(\partial)=\chi(X)+r.
$
Therefore,
\begin{equation}\label{eq1}
\chi(S)=\chi(A)+\chi(X)+r.
\end{equation}

Next, we compute $\chi(S)$ in a different way. Let $X^*\subset X$ be the maximal non-empty open subset over which
\[
p|_{p^{-1}(X^*)}:U:=p^{-1}(X^*)\longrightarrow X^*
\]
is smooth. By Ehresmann's theorem, all the fibers over $X^*$ are diffeomorphic. Let $F$ denote the fiber over a point $x\in X^*$. Since we are working over the complex numbers,
\[
\chi(U)=\chi(F)\chi(X^*).
\]
Hence,
\[
\chi(S)=\chi(F)\chi(X^*)+\sum_{c\in\Delta}\chi(F_c),
\]
where
\[
\Delta=X\setminus X^*,\qquad F_c=p^{-1}(c).
\]

Adding and subtracting $\chi(F)$ for each $c\in\Delta$, we obtain
\begin{align}\label{eq3}
\chi(S)
&=\chi(F)\left(\chi(X^*)+\sum_{c\in\Delta}1\right)
+\sum_{c\in\Delta}\bigl(\chi(F_c)-\chi(F)\bigr)\\
&=\chi(F)\chi(X)+\sum_{c\in\Delta}(\delta_c+r_c-1),
\end{align}
where
\[
\delta_c=\dim H^1(F,\QBbb)-\dim H^1(F_c,\QBbb),
\]
and $r_c$ denotes the number of irreducible components of $F_c$.

If any smooth fiber of the map $p$ is contained in $\partial$ then it must be isomorphic to $\PBbb^1$. Thus the general fiber would be $\PBbb^1$. So we assume none of the smooth fiber is contained in $\partial$.
Let
\[
\Delta'=\{c\in\Delta:\;p(L\cong\PBbb^1)=c,\;L\subset\partial\},
\]
and for each $c\in\Delta'$, let $r'_c$ denote the number of irreducible $\PBbb^1$-components of $\partial$ contained in $p^{-1}(c)$. Clearly, $\Delta'\subset \Delta$ and
$
r'_c\leq r_c.
$
Since the genus of $X$ is positive, we have
$
\sum_{c\in\Delta'}r'_c=r.
$
Moreover, let
\[
\Delta''=\{c\in\Delta':r_c=r'_c\}.
\]
Combining (\ref{eq1}) and (\ref{eq3}), we obtain 
\iffalse$\chi(A)+\chi(X)+r
=\chi(F)\chi(X)+\sum_{c\in\Delta}(\delta_c+r_c-1).$ Let's focus on the affine case first. 
\fi
\begin{align*}
\chi(A)+\chi(X)+r
&=\chi(F)\chi(X)+\sum_{c\in\Delta}(\delta_c+r_c-1)\\
\Longrightarrow\qquad
\sum_{c\in\Delta'}r'_c+\chi(A)-\chi(X)
&=\chi(X)(\chi(F)-2)
+\sum_{c\in\Delta'}(\delta_c+r_c-1)
+\sum_{c\in(\Delta-\Delta')}(\delta_c+r_c-1)\\
\Longrightarrow
\chi(A)-\chi(X)
&=4g_F(g-1)
+\sum_{c\in\Delta''}(\delta_c-1)
+\sum_{c\in(\Delta'-\Delta'')}(\delta_c+r_c-r'_c-1)
\\&\qquad+\sum_{c\in(\Delta-\Delta')}(\delta_c+r_c-1).
\end{align*}

Now Theorem~\ref{mainthm} implies that $\delta_c\geq0$ for every $c\in\Delta$. Moreover,
\[
4g_F(g-1)\geq0,\quad r_c-r'_c-1\geq0
\]
for every $c\in\Delta'-\Delta''$.

Furthermore, if $c\in\Delta''$, then the reduced fiber $F_c^{\mathrm{red}}$ is a tree of $\PBbb^1$'s. Hence
$
\dim H^1(F_c,\QBbb)=0.
$
Therefore, if $g_F\neq0$, then
$
\delta_c\geq2
$
for every $c\in\Delta''$. Consequently,
$
\sum_{c\in\Delta''}(\delta_c-1)\geq\sum_{c\in\Delta''}1.
$

 This means that R.H.S. of the above equation is non-negative. Now first consider the affine case. Let $\dim(gr^W_2(H^1(A,\QBbb))=a$, then \[\chi(A)=1-(a+2g)=2-2g-(a+1)=\chi(X)-(a+1).\] Thus the L.H.S. of the equation obtained from combining \ref{eq1} and \ref{eq3} is at most $-1$. This is a contradiction. So in this case the general fiber is $\PBbb^1$. 

 Now let's consider the proper case. Here $\chi(A)=\chi(X)$, as a consequence using the similar computations as before we get,
\[
0
=4g_F(g-1)
+\sum_{c\in\Delta''}(\delta_c-1)
+\sum_{c\in(\Delta'-\Delta'')}(\delta_c+r_c-r'_c-1)
+\sum_{c\in(\Delta-\Delta')}(\delta_c+r_c-1).
\]

We now distinguish two cases.

\begin{itemize}
\item[$g\geq2$:]
Assume that $g_F\neq0$. Then
\[
4g_F(g-1)>0,
\]
while every other summand on the right-hand side of the above equation is non-negative. Hence the entire right-hand side is strictly positive, whereas the left-hand side is $0$, a contradiction. Therefore,
\[
g_F=0.
\]

\item[$g=1$:]
In this case, the above equality becomes
\begin{equation}
\sum_{c\in\Delta''}(2g_F-1)
+\sum_{c\in(\Delta'-\Delta'')}(\delta_c+r_c-r'_c-1)
+\sum_{c\in(\Delta-\Delta')}(\delta_c+r_c-1)
=0.
\end{equation}

Since $g_F>0$, this immediately implies that
$
\Delta''=\varnothing,\;
\delta_c=0
$
for every $c\in\Delta$. Hence
\[
\dim H^1(F,\QBbb)=\dim H^1(F_c,\QBbb)
\]
for every $c\in\Delta$. By Remark~\ref{rem 2}, it follows that
$
R^1p_*(\underline{\QBbb}_S)
$
is a local system.

Now consider the base change of $p$ by the universal covering map
\[
u:\CBbb\longrightarrow X.
\]
We obtain the following Cartesian diagram:
\[
\begin{tikzcd}
	{S'} & S \\
	{\mathbb{C}} & X
	\arrow["{u'}", from=1-1, to=1-2]
	\arrow["{p'}"', from=1-1, to=2-1]
	\arrow["p", from=1-2, to=2-2]
	\arrow["u"', from=2-1, to=2-2].
\end{tikzcd}
\]

Since $u$ is a covering map,
$
R^1p'_*(\underline{\QBbb}_{S'})
$
is a local system of the same rank as
$
R^1p_*(\underline{\QBbb}_S),
$
namely $2g_F$. Since $\CBbb$ is simply connected, this local system is trivial. Therefore, the Leray spectral sequence for $p'$ yields
$
\dim H^1(S',\QBbb)=2g_F>0.
$
On the other hand, the above fiber diagram induces the following commutative diagram of fundamental groups:
\[
\begin{tikzcd}
	{\pi_1(S')} & {\pi_1(S)} \\
	{(e)=\pi_1(\mathbb{C})} & {\pi_1(X)}
	\arrow["{u'_*}", hook, from=1-1, to=1-2]
	\arrow["{p'_*}"', from=1-1, to=2-1]
	\arrow["{p_*}", from=1-2, to=2-2]
	\arrow["u_*"', from=2-1, to=2-2].
\end{tikzcd}
\]

The map $u'_*$ is injective because $u'$ is a covering map (being the base change of $u$). Moreover, by Corollary~\ref{fundamentalgroup}, the map
$
p_*:\pi_1(S)\longrightarrow\pi_1(X)
$
is an isomorphism. Hence $\pi_1(S')$ is trivial. By the universal coefficient theorem,
\[
H^1(S',\QBbb)=0,
\]
which contradicts the equality
$
\dim H^1(S',\QBbb)=2g_F.
$
Therefore,
$
g_F=0.
$
\end{itemize}
\end{proof}

Thus, $S$ is birational to $\PBbb^1\times X$.
\end{proposition}
\begin{remark}
The proof of the above proposition is analytic in nature in the case $g=1$, as it uses the universal cover. We also note that the product formula for Euler characteristics used in the proof for a smooth proper fibration does not hold in general over fields of arbitrary characteristic (cf. \cite[expose X]{SGA5}).
\end{remark}

\begin{theorem}\label{minimal}
Let $S$ be a smooth projective surface birational to $\PBbb^1\times X$, where $X$ is a smooth projective curve of positive genus. Assume that there exists a morphism
\[
p:S\to X
\]
with connected fibers. Moreover, let $A$ be an affine open subset of $S$ satisfying Condition~J. Then the following are true: \begin{enumerate}
    \item There exists an open subset $C\subset X$ such that $A\twoheadrightarrow C\hookrightarrow X$. Any fiber $F$ of the surjective map $A\to C$ satisfies $F^{\text{red}}\cong\ABbb^1$, with general fiber being smooth $\ABbb^1$. There can be finitely many multiple fibers.
    \item In the proper case $C=X$. In the affine case the cardinality of the set $X\setminus C$ is \[1+\dim(gr^W_2(H^1(A,\QBbb))\].
    \item If the map $A\to C$ is smooth then $S$ is minimal.
\end{enumerate}
\begin{proof}
Let
\[
f:S\to T
\]
be the contraction to a minimal model. Then $T$ is a projective bundle over $X$, and we have the following commutative diagram:
\[
\begin{tikzcd}
	S && T \\
	& X
	\arrow["f", from=1-1, to=1-3]
	\arrow["p"', from=1-1, to=2-2]
	\arrow["q", from=1-3, to=2-2].
\end{tikzcd}
\]
Here $q$ is the projective bundle and $f$ is a composition of blow-ups at closed points.

Since the genus of $X$ is positive, every irreducible component of $\partial$ other than $\partial'$ maps to a point of $X$; equivalently, it is contained in a fiber over a closed point of $X$. From the above diagram, it follows that reduced subscheme of every fiber of $p$ is a tree of $\PBbb^1$'s. Moreover, for every $c\in X$, the curve $\partial'$ intersects exactly one irreducible component of $F_c:=p^{-1}(c)$, and the intersection consists of a single point. 
\iffalse Also from remark \ref{surjection} it is clear that none of the smooth fiber $p$ is contained in $\partial$ and each singular fiber has at least one irreducible component which is not contained in $\partial$.    
\fi
\iffalse
Let $D$ be a leaf or a bridge of $\partial$ isomorphic to $\PBbb^1$, and let $p(D)=c\in C$. We claim that for each such $D$ which is also a leaf of the fiber
$
p^{-1}(c)=:F_c
$ then there exist a leaf of $F_c$ which is not a component of $\partial$. Otherwise there are two possibilities, either it is obtained from blowing up a smooth point in the fiber over $c$ in a intermediate stage $T'$ between $S$ and $T$ or a node of it and any further blow-up whose center lies in it must be at the point at which it connects to the rest of the fiber. In the case there are further blow-ups then those are also part of $\partial$ (by connectivity of $\partial$ and 
Indeed, otherwise
This means $D$ is a minus one curve, which contradicts condition $J$.
Notice that the same argument shows that every leaf of every fiber $F_c$ has self-intersection $-1$.
\fi
Next, observe that any connected subtree
$
\Theta\subset\overline{(F_c)^{\text{red}}\setminus\partial}
$
must be irreducible.

Indeed, we first claim that
$
\Theta\cap\partial
$
consists of a single point. Otherwise,
$
\Theta\cap\overline{(F_c)^{\text{red}}\setminus\Theta}
$
would contain at least two points, contradicting the fact that $F_c$ is a connected tree. Moreover, this intersection point is necessarily a smooth point of $\Theta$.
Now suppose that $\Theta$ has more than one irreducible component. Then one of its irreducible components does not meet $\partial$. Consequently, this component is contained entirely in $A$, contradicting the fact that $A$ is affine. This proves the claim. In fact we have proved that $(F_c)^{\text{red}}\setminus\partial$ is a disjoint union of $\ABbb^1$'s. 

Let the composite map $A\hookrightarrow S\xrightarrow{p} X$ be $p'$ and the image of $A$ in $X$ be $C$. Then $C$ is an open subset of $X$. By previous discussion it is clear that any smooth fiber of the map $p'$ is isomorphic to $\ABbb^1$ and singular fibers are disjoint union of (possibly unreduced) $\ABbb^1$'s. Let $\{c_1,...,c_k\}$ be the points of $C$ over which $p'$ is singular. Then \[p'^{-1}(c_i)=\sqcup^{n_i}_{j=1} m^j_i\cdot D^i_j,\;1\leq i\leq k,\] where each $D^i_j$ is isomorphic to $\ABbb^1$ and $m^j_i$ is the multiplicity of $D^i_j$. Next we want to compute the stalks of the constructible sheaves $R^lp'_*(\underline{\ZBbb}_A)$ for $0\leq l\leq 2$. It is clear that for a point $c\in C\setminus\{c_1,...,c_k\}$,\[R^0p'_*(\underline{\ZBbb}_A)_c=\ZBbb,\; R^1p'_*(\underline{\ZBbb}_A)_c=0,\; R^2p'_*(\underline{\ZBbb}_A)_c=0.\]
Now for a point $c_i\in\{c_1,...,c_k\}$ consider a small analytic disc $D$ around it in $C$. Let's define $D^*=D\setminus \{c_i\}$. We also define $U:=p'^{-1}(D)$ and $U^*:=p'^{-1}(D^*)$. Thus $U^*=U\setminus\sqcup_{j=1}^{n_i}D^i_j$. It is clear that $H^*(U^*,\ZBbb)\cong H^*(D^*,\ZBbb)$. The gysin sequence then gives us:\[0\to H^1(U,\ZBbb)\to H^1(U^*,\ZBbb)\cong\ZBbb\xrightarrow{1\mapsto(m^1_i,...,m^{n_i}_i)}\bigoplus_{j=1}^{n_i}(H^0(D^i_j,\ZBbb)\cong\ZBbb)\to H^2(U,\ZBbb)\to0.\]
Thus $H^1(U,\ZBbb)=0$ and $H^2(U,\ZBbb)=(\bigoplus_{j=1}^{n_i}\ZBbb)/\ZBbb(m^1_i,...,m^{n_i}_i)=:G_i$. It is also easy to see that $H^0(U,\ZBbb)=\ZBbb$. Thus we get that \[R^0p'_*(\underline{\ZBbb}_A)=\underline{\ZBbb}_{C},\; R^1p'_*(\underline{\ZBbb}_A)=0,\;R^2p'_*(\underline{\ZBbb}_A)=\bigoplus_{i=1}^k S_i.\]
Where $S_i$ is a skyscraper sheaf supported at $c_i$ with the underlying group $G_i$. Then the Leray spectral sequence gives us: $H^0(A,\ZBbb)=\ZBbb,\; H^1(A,\ZBbb)=H^1(C,\ZBbb)$ and \[0\to H^2(C,\ZBbb)\to H^2(A,\ZBbb)\to\bigoplus^k_{i=1}G_i\to0.\] Here the leftmost map is the natural pullback map $H^2(C,\ZBbb)\to H^2(A,\ZBbb)$.
Next note that each $G_i$ is an abelian group of rank $n_i-1$. Thus universal coefficient theorem gives us $H^2(A,\QBbb)=H^2(C,\QBbb)\oplus\QBbb^{\sum^k_{i=1}(n_i-1)}$. Now in the proper case by remark \ref{surjection} we know that $C=X$, so $H^2(C,\QBbb)\cong\QBbb$. On the other hand $H^2(A,\QBbb)\cong\QBbb$, hence $n_i=1$ for all $i$. In the affine case $H^2(A,\QBbb)=0$, hence $C\neq X$ and $n_i=1$ for all $i$. Now $n_i=1$ for all $i$ implies $G_i\cong\ZBbb/m_i\ZBbb$.
We therefore obtain the following conclusions.

\begin{enumerate}\label{1}
\item Any fiber $F$ of the map $A\to C$ satisfies $F^{\text{red}}\cong\ABbb^1$. The general fiber is smooth and there are finitely many multiple fibers. Here $C$ is an open subset of $X$ which is the image of $A$. Moreover in the proper case $C=X$ and in the affine case $C\subsetneq X$. 

\item Suppose $\{c_1,...,c_k\}$ be the points of $C$ over which the map $A\to C$ has multiple fibers with multiplicities $m_1,...,m_k$ respectively then,\[H^1(C,\ZBbb)\xrightarrow{\cong}H^1(A,\ZBbb).\] Thus in the affine case $gr^W_2(H^1(A,\QBbb))=gr^W_2(H^1(C,\QBbb))=s-1$, where $s$ is the cardinality of the set $X\setminus C$ and
 $H^2(A,\ZBbb)=\bigoplus_{i=1}^k\ZBbb/m_i\ZBbb$. In the proper case \[0\to H^2(C,\ZBbb)\cong\ZBbb\to H^2(A,\ZBbb)\to\bigoplus_{i=1}^k\ZBbb/{m_i}\ZBbb\to0.\]
\end{enumerate}
These prove the first two parts of the theorem.
For the last part we study the structure of the singular fiber more closely. Let's write $f=g_N\circ ...\circ g_1$,
\[\begin{tikzcd}
	{S=S_N} & {...} & {S_1} & {S_0=T}
	\arrow["{g_N}", from=1-1, to=1-2]
	\arrow["{g_2}", from=1-2, to=1-3]
	\arrow["{g_1}", from=1-3, to=1-4],
\end{tikzcd}\]
where each $f_i$ is a blow-up at a single smooth point of $S_{i-1}$, for $1\leq i\leq N$. We first claim that no singular fiber of the map $p$ can be contained in $\partial$. If possible assume otherwise. Let this singular fiber be $F_c=p^{-1}(c)$. Then there exists a $1<j\leq N$ such that the center of blow-up of $g_j$ lies over $c$ and for all $i>j$ the centers don't map to $c$. Now the blow-up $g_j$ induces a leaf or a bridge in $F_c$ which is a $-1$ curve but as $(F_c)^{\text{red}}\subset\partial$ then this component is a $-1$ leaf or a bridge of $\partial$ too which is impossible. Thus the collection of singular fibers of the map $p$ is $\{c_1,...,c_k\}$.

Let's denote $f_i=g_i\circ f_{i-1}$ for $2\leq i\leq N$ and $f_1=g_1$. We will study a singular fiber $F_c$, where $c\in\{c_1,...,c_k\}$. It is easy to see that if for all $1\leq i\leq N$, the center of blow-up for $g_i$ is not a node point of $f^{-1}_{i-1}(c)$ then $F_c$ is reduced. We claim that then any leaf $D$ of $F_c$ must be a $-1$ curve. Indeed $
D\cdot\overline{(F_c\setminus D)}=1.
$
But $D$ being an irreducible component of the fiber $F_c$ gives us $D\cdot F_c=0$, thus
\[
D\cdot D=-D\cdot\overline{(F_c\setminus D)}=-1.
\]
Thus any leaf of $F_c$ cannot be a leaf or a bridge of $\partial$. By previous discussion we know that $F_c\setminus\partial$ consists of single irreducible $\ABbb^1$. This implies any leaf of $F_c\cap\partial$ is a leaf of $F_c$ and a leaf or a bridge of $\partial$. Hence $\partial$ contains a leaf or a bridge which is a $-1$ curve. This is a contradiction. 

Thus there exists a $g_j$ such that its center of blow-up is a node point of $f^{-1}_{j-1}(c)$. Note that the exceptional divisor one gets after blowing-up a node point has multiplicity in the fiber. Call this exceptional divisor $E$. Let $\tilde{E}$ be the strict transform of $E$ in $S$. If for all $i>j$ the center of blow-up for $g_i$ is disjoint from strict transforms of $E$ then $\tilde{E}$ is a bridge of $F_c$ hence either a bridge or leaf of $\partial$. Moreover $\tilde{E}$ has self intersection -1 in $S$. Indeed let $g:=g_N\circ...\circ g_{j+1}$ then \[F_c=g^{-1}(f^{-1}_j(c))=\tilde{(f^{-1}_j(c))}+E',\]
where $\tilde{(f^{-1}_j(c))}$ is the strict transform and $E'$ is the exceptional divisor of $g$. As $\tilde{E}\cdot E'=0$, we have $\tilde{E}^2=E^2=-1$. Which is impossible. So there is some center on $E$ which has been blown up. Next we make two simple observations:
\begin{enumerate}
    \item each blow-up adds an leaf or a bridge,
    \item if for some $j$, $f^{-1}_j(c)=m D+F'$, for some effective divisor $F'$, reduced divisor $D$, integer $m>1$ and $g_{j+1}$ is blows-up a point in $D$ then the multiplicity of the exceptional divisor $E'$ of $g_{j+1}$ is at least $m$ in $f^{-1}_{j+1}(c)$.
\end{enumerate}
The first one is obvious. For the second point we know that $f^{-1}_{j+1}(c)=aE'+m\tilde{D}+\tilde{F'}.$
Here $\tilde{D},\tilde{F'}$ denotes the respective strict transforms. Then this being a fiber for $f_{j+1}$ gives us, \begin{align*}
E'\cdot(aE'+m\tilde{D}+\tilde{F'})&=0\\
\implies a(E')^2+m(E'\cdot \tilde{D})+E'\cdot\tilde{F'}&=0\\
\implies-a+m&\leq0\\
\implies a&\geq m.
\end{align*}
Combining these observations we get that the integer $i$, for which the center of blow-up $g_i$ is in the strict transform of $E$ adds a leaf or a bridge to the fiber $F_c$ and also this component is unreduced. Now this cannot be a bridge otherwise it would be a $-1$ bridge for $\partial$. So it must be a leaf but then it cannot be contained in $\partial$. Thus this is the component which intersects $A$ non-trivially. Hence the fiber of the map $A\to C$ over $c$ has multiplicity. This is true for all points $c\in\{c_1,...,c_k\}$. This implies if the map $A\to C$ is smooth then the set $\{c_1,...,c_k\}$ must be empty. In other words $S=T$, which is minimal. We also conclude from this that in the proper case $\partial=\partial'$ and in the affine case $\partial$ is union of $\partial'$ and the fibers (which are smooth) over the points $X\setminus C$. 
\end{proof}
\end{theorem}
Before proving the main theorem we recall a result from \cite{Xiao} about fundamental group of a surface fibering over a curve.
\begin{proposition}[Lemma 2 of \cite{Xiao}]\label{orbifold}
    Let $f:X\to C$ be a morphism from a smooth surface $X$ to a smooth curve $C$. The general fiber of $f$ is smooth connected (Let's call it $F$). There are finitely many singular fibers $F_1,..., F_k$ over the points $p_1,...,p_k$. Then $F_i=\sum^{n_i}_{j=1}m^i_j\cdot D^i_j$ for irreducible reduced divisors $D^i_j$ in $X$. Let $m_i=\gcd(m^i_j|1\leq j\leq n_i)$. Then we have the following sequence of fundamental groups:\[\pi_1(F)\to\pi_1(X)\to\Gamma\to 1.\]
    Here $\Gamma$ has the following presentation:\[\Gamma=\langle a_1,b_1,...,a_g,b_g,x_1,...,x_s,\gamma_1,...,\gamma_k\rangle/\langle\prod^g_{i=1}[a_i,b_i]\cdot\prod^{s}_{i=1}x_i\cdot\prod^{k}_{i=1}\gamma_i,\;\gamma^{m_1}_1,...,\gamma^{m_k}_k\rangle.\]
    Where $g$ is the genus of a smooth compactification $\Bar{C}$ of $C$, with $\Bar{C}\setminus C=\{q_1,...,q_s\}$.
\end{proposition}
Combining the above results, we obtain the following theorem.

\begin{theorem}[Theorem~\ref{thetheorem}]
Let $A$ be a smooth affine surface such that there is an isomorphism of $\QBbb$-mixed Hodge structures (as cohomology rings)
\[
H^*(A,\QBbb)\cong H^*(C,\QBbb),
\]
where $C$ is a smooth curve with compactification $\bar C$ of genus $g\geq1$ and satisfies one of the following:
\begin{enumerate}
    \item $H^*(A,\ZBbb)\cong H^*(C,\ZBbb)$ as cohomology rings,
    \item $\pi_1(A)$ is torsion free.
\end{enumerate} Then $A$ admits a smooth fibration over $C$ whose fibers are isomorphic to $\ABbb^1$.

\begin{proof}
Choose a smooth projective compactification $S$ of $A$ satisfying Condition $J$. $S$ fibers over a smooth projective curve $X$ of genus $g$ with general fiber being $\PBbb^1$. $C$ be the image of $A$ in $X$. Using first part of theorem~\ref{minimal}, we know that the surjective map $A\to C$ has the following property: each fiber $F$ satisfies $F^{\text{red}}\cong\ABbb^1$, with possibly finitely many unreduced fibers. If possible $F_1,...,F_k$ be the unreduced fibers with multiplicities $m_1,...,m_k$. Clearly $m_i\geq 2$ for $1\leq i\leq k$. First assume $(1)$ of this theorem. Then in the affine case $H^2(A,\ZBbb)=0$. On the other hand by theorem~\ref{minimal}, $H^2(A,\ZBbb)\cong\bigoplus_{i=1}^k\ZBbb/m_i\ZBbb$. Thus in this case there can be no multiple fibers. In the proper case again by theorem~\ref{minimal}, $C=X$ and $H^1(A,\ZBbb)\cong H^1(C,\ZBbb)$ via the natural pull-back map. By the condition that abstract ring isomorphism $H^*(A,\ZBbb)\cong H^*(C,\ZBbb)$ one gets that $H^2(A,\ZBbb)\cong\ZBbb$ and the cup product pairing in $H^*(A,\ZBbb)$ is unimodular. Thus the pull-back map $\ZBbb\cong H^2(C,\ZBbb)\to H^2(A,\ZBbb)\cong\ZBbb$ must be an isomorphism. While in theorem~\ref{minimal} we get this pullback map satisfies the following short exact sequence, \[0\to H^2(C,\ZBbb)\to H^2(A,\ZBbb)\to\bigoplus_{i=1}^k\ZBbb/m_i\ZBbb\to0.\] This implies that there can be no multiple fiber in this case either so $A\to C$ is a smooth fibration with fibers isomorphic to $\ABbb^1$.

Next we assume the condition $(2)$ i.e. $\pi_1(A)$ is torsion free. Then by proposition~\ref{orbifold} \[\pi_1(A)=\langle a_1,b_1,...,a_g,b_g,x_1,...,x_s,\gamma_1,...,\gamma_k\rangle/\langle\prod^g_{i=1}[a_i,b_i]\cdot\prod^{s}_{i=1}x_i\cdot\prod^{k}_{i=1}\gamma_i,\;\gamma^{m_1}_1,...,\gamma^{m_k}_k\rangle.\] Here $a_i,b_i$'s are generators of $\pi_1(X)$ and $s$ is the cardinality of the set $X\setminus C$. If $X=C$ then the relations in the above presentation has no $x_i$ term. As $g\geq 1$ we get from \cite{Fuchsian} that orders of the elements $\gamma_i$'s are exactly $m_i>1$. But $\pi_1(A)$ is torsion free, this is a contradiction. Hence all the fibers of the map $A\to C$ are reduced. The result now follows immediately.
\end{proof}
\end{theorem}
\begin{remark}
     In fact in the affine case the torsion $\pi_1(A)$ induces torsion in $H_1(A,\ZBbb)$, hence in $H^2(A,\ZBbb)$. Indeed abelianisation of $\pi_1(A)$ in this case has the following presentation:\[0\to\ZBbb\xrightarrow{1\mapsto(x_1+...+x_s+\gamma_1+...+\gamma_k)}\ZBbb^{\oplus 2g}\oplus\bigoplus_{i=1}^s\ZBbb\langle x_i\rangle\oplus\bigoplus_{i=1}^k(\ZBbb/m_i\ZBbb)\langle\gamma_i\rangle\to H_1(A,\ZBbb)\to0.\] Now $H_1(A,\ZBbb)$ cannot be torsion free otherwise the above sequence will split, which will force $m_i$'s to be 1. On the other hand it is not true that in proper case existence of torsion in $\pi_1$ implies existence of torsion in $H_1$ (cf. example \ref{example}).
\end{remark}
     It is worth noting that with just the rational M.H.S. isomorphism (as graded vector spaces) in the proper case $A\to X$( as in theorem~\ref{minimal}) may not be surjective. But the computations the aforementioned theorem suggests, in that case the image $C$ misses exactly one point and there exists a fiber which is a disjoint union of two affine lines. Then by proposition~\ref{fundamentalgroup}, \[\pi_1(A)=\langle a_1,b_1,...,a_g,b_g,x,\gamma_1,...,\gamma_k\rangle/\langle\prod^g_{i=1}[a_i,b_i]\cdot x\cdot\prod^{k}_{i=1}\gamma_i,\;\gamma^{m_1}_1,...,\gamma^{m_k}_k\rangle.\] It is clear that this group has a surjection to the free group in $2g$ many generators. Now if we assume $\pi_1(A)$ is isomorphic to fundamental group of a genus $g$ curve then this is impossible so $C=X$. Also in that case we have $\pi_1(A)$ has no torsion so we get that $A\to C$ is smooth by previous theorem. Clearly in the affine case this condition is sufficient. Thus we get the following result:
     \begin{theorem}\label{thm2}
         Let $A$ be a smooth affine surface such that there is an isomorphism of $\QBbb$-mixed Hodge structures (as graded $\QBbb$-vector spaces)
\[
H^*(A,\QBbb)\cong H^*(C,\QBbb),
\]
where $C$ is a smooth curve with compactification $\bar C$ of genus $g\geq1$ and $\pi_1(A)\cong\pi_1(C)$ then $A$ admits a smooth fibration over $C$ whose fibers are isomorphic to $\ABbb^1$.
     \end{theorem}

In the following proposition we give a classification of all smooth affine surfaces as in theorem~\ref{thetheorem}. 
\begin{proposition}\label{classification}
  There is a one-to-one correspondence between smooth affine surfaces as in theorem \ref{thetheorem} and the following data:
  \begin{itemize}
      \item In the affine case:
      \begin{itemize}
          \item[(a)] a smooth affine curve $C$ such that its smooth compactification has positive genus,
          \item[(b)] a line bundle $L$ on $C$.
      \end{itemize}
  
  \item In the proper case: 
  \begin{itemize}
\item[(a)] a smooth projective curve $C$ of positive genus,
\item[(b)] a line bundle $L$ on $C$ with $\deg(L)>0$, and
\item[(c)] an extension class in $\PBbb(H^1(C,L^{\vee}))$.
\end{itemize}
\end{itemize}
\begin{proof} Let $A$ be such an affine surface. Then by theorem \ref{minimal} and theorem \ref{mainthm} there is a projective completion of $A$ which is  projectivization of a rank two vector bundle, say $E$, over $\bar{C}$, which is the smooth compactification of $C$. The boundary (in both affine and proper case) contains a section, which corresponds to a line bundle quotient $E\twoheadrightarrow M$. It is then easy to see that the complement of this section is a torsor under the line bundle
\[
M^{\vee}\otimes N,
\]
where $N$ is the kernel of the surjection $E\twoheadrightarrow M$. After twisting $E$ by $N^{\vee}$, we may assume that $E$ fits into an extension
\[
0\to\Ocal_{\bar{C}}\to E\to L\to0,
\]
and the corresponding extension class coincides with the class of the above torsor. Note that $L^{\vee}=M^{\vee}\otimes N$. In the affine case the boundary contains the fibers over $\bar{C}\setminus C$. Thus $A$ is this torsor restricted to $C$. As $C$ is affine so the torsor is trivial hence it is the total space of the line bundle $L^{\vee}|_{C}$. On the other hand total space of any line bundle on $C$ gives us an affine surface with required properties. Next we focus on the proper case. Here $C=\bar{C}$ and the boundary divisor is just the section. Since $A$ is affine, this torsor cannot be trivial. By \cite[Th. 2.4]{Hartshorne}, it follows that $\deg(L)>0$, and moreover every non-trivial extension of $\Ocal_C$ by such a line bundle $L$ gives rise to an affine surface. Indeed by the aforementioned theorem one gets that such an extension is an ample vector bundle. As a result the section which corresponds to the relative $\Ocal(1)$ of the projectivization is ample so its complement is affine.

Furthermore, we have shown that the morphism $A\to C$ can be recovered intrinsically from $A$. The unipotent radical of the automorphism group of this $\ABbb^1$-fibration is naturally isomorphic to the additive group of the line bundle $L^{\vee}$, thereby recovering $L$ itself. The $L^{\vee}$ torsor $A\to C$ determines the class in $\PBbb(H^1(C,L^{\vee}))$. Hence the result follows.
\end{proof}
\end{proposition}
\begin{remark}
    It is evident from theorem~\ref{minimal} that isomorphism in rational M.H.S.(as rings) guarantees a surjection from $A\to C$ such that each  fiber has underlying reduced structure isomorphic to $\ABbb^1$ and finitely many fibers with multiplicities. Then by  \cite[Th. 11]{Dubouloz} the map $A\to C$ factors through a smooth algebraic space $\hat{C}$ (not necessarily locally separated). Moreover the map $A\to \hat{C}$ is a \'etale torsor for an \'etale line bundle $L$ on $\hat{C}$. In the case $A\to C$ is smooth $\hat{C}=C$.
    \end{remark}
\begin{example}\label{example}
    Let $C$ be a smooth projective curve with $\Ecal$ a rank two vector bundle on $C$. Let $X=\PBbb(\Ecal)$ and $p:X\to C$ be the projection map. Furthermore assume that $p$ has a section $\Delta$ which is an ample divisor in $X$. Let $c$ be a point in $C$ and $F$ be the fiber $p^{-1}(c)$. Consider the following sequence of blow-ups at closed points:
    \[\begin{tikzcd}
	{X_3=\text{Bl}_zX_2} & {X_2=\text{Bl}_yX_1} & {X_1=\text{Bl}_x X} & X
	\arrow["{q_3}", from=1-1, to=1-2]
	\arrow["{q_2}", from=1-2, to=1-3]
	\arrow["{q_1}", from=1-3, to=1-4].
\end{tikzcd}\]
Let $E_i$ be the exceptional divisor of the map $q_i$ for $1\leq i\leq 3$. Here $x$ is a closed point in $F\setminus\Delta$, $y$ is the point of intersection of strict transform of $F$ in $X_1$, $E_1$ and $z$ is a closed point in $E_2$ which is not in the union of strict transform of $F$ and $E_1$ in $X_2$. Let's denote $p_1=q_1\circ p,\;p_2=q_2\circ p_1$ and $p_3=q_3\circ p_2$. For brevity we denote the strict transforms of the divisors $F,E_1,E_2$ by the same letters.

Now it is clear that $p^{-1}_1(c)=F+E_1$. Then $p_2^{-1}(c)=q^{-1}_2(F+E_1)$. This has the following form $F+E_1+aE_2$ for some integer $a$. We know that $E_2\cdot F=1= E_2\cdot E_1$ and $E^2_2=-1$. Also $F+E_1+aE_2$ is fiber of the map $p_2$ thus,\[E_2\cdot(F+E_1+aE_2)=0\implies a=2.\]
Similarly we can say $p^{-1}_3(c)=F+E_1+2E_2+bE_3$ for some integer $b$. Counting the intersections we get $E_3\cdot E_2=1$, $E_3\cdot F=0=E_3\cdot E_1$ and $E^2_3=-1$. Again as before,
\[E_3\cdot(F+E_1+2E_2+bE_3)=0\implies b=2.\]
Thus $p^{-1}_3(c)=F+E_1+2E_2+2E_3$ and in $X_3$ we have $E_2\cdot E_1=E_2\cdot E_3=E_2\cdot F=1$ with $E_1^2=E_2^2=F^2=-2$ and $E^2_3=-1$. Let $\Delta'$ be the strict transform of $\Delta$ in $X_3$, then $\Delta'\cdot E_i=0$ for $1\leq i\leq 3$ and $\Delta'\cdot F=1$. It is easy to see that $\Delta'^2=\Delta^2=\alpha>0$.

Let's consider the the open subvariety $U$ of $X_3$ which the complement of the divisor $D':=\Delta'\cup F\cup E_1\cup E_2$. By the above discussion it is clear that this maps surjectively onto $C$ and the map is a smooth $\ABbb^1$ fibration over $C\setminus\{c\}$. The fiber over $c$ is an unreduced $\ABbb^1$ with multiplicity $2$. Clearly $U$ is smooth, we claim that $U$ is also affine. It is enough to show that $D'$ supports an ample divisor. Consider the divisor $D=a_1\Delta'+a_2F+a_3E_1+a_4E_2$ for positive integers $a_1,a_2,a_3,a_4$. Then by the above intersection computations we get,
\begin{align*}
    D^2&=2((\alpha/2) a_1^2-a_2^2-a^2_3-a^2_4+a_1a_2+a_2a_3+a_3a_4),\\
     D\cdot \Delta'&=\alpha a_1+a_2,\\
     D\cdot F&=a_1-2a_2+a_3,\\
     D\cdot E_1&=a_2-2a_3+a_4,\\
     D\cdot E_2&=a_3-2a_4.
\end{align*}
    Note that $\alpha a_1+a_2$ is always positive. Choosing $a_4=1,a_3=3,a_2=6$ and $a_1=10$ makes the last three intersection numbers positive. Moreover $D^2>2(a^2_1/2-a^2_2-a^2_3-a^2_4)$. Which for our choice is $2(50-36-9-1)=8>0$. It is easy to see that for any irreducible curve in $X_3$ which is either a fiber of some point $c'\neq c$ or maps surjectively to $C$ intersects the support of $D$. Hence by the Nakai-Moishezon criterion of ampleness $D$ is an ample divisor. Note that the pullback map $H^*(C,\QBbb)\to H^*(A,\QBbb)$ is an isomorphism. Let the genus of $C$ be $g\geq1$ then by proposition \ref{orbifold} and \cite{Fuchsian} we know that $\pi_1(A)=\langle a_1,b_1,...,,a_g,b_g,\gamma\rangle/\langle \prod_{i=1}^g[a_i,b_i]\cdot\gamma,\gamma^2\rangle$, where the element $\gamma$ has order exactly 2. Thus $\pi_1(A)$ is not torsion free. On the other hand taking the abelianization we get $H_1(A,\ZBbb)=\ZBbb^{2g}$ so even in $\ZBbb$ coefficients $H^*(A,\ZBbb)\cong H^*(C,\ZBbb)$ (this isomorphism is abstract isomorphism of graded $\ZBbb$ modules, not cohomology rings). In this case the pull-back map $H^2(C,\ZBbb)\to H^2(A,\ZBbb)$ is $\ZBbb\xrightarrow{1\mapsto2}\ZBbb$. Clearly the cup-product matrix is not unimodular in this case. It is worthwhile to note that the boundary divisor $D'$ of $A$ in $X_3$ has no leaf or bridge which is a $-1$ curve. 
\end{example}
\begin{example}\label{example2}
    Let's do the same example for the base curve is $\PBbb^1$. Then the affine surface $A$ has a surjective map to $\PBbb^1$ with a single unreduced fiber (say over $\infty\in\PBbb^1$) of multiplicity $2$ and all fibers have reduced structure isomorphic to $\ABbb^1$. Clearly by computations as in~\ref{minimal} gives us that the natural pull-back map induces the following isomorphism $H^*(A,\QBbb)\cong H^*(\PBbb^1,\QBbb)$. Note that $V$, the complement of the unreduced fiber is a line bundle torsor over $\PBbb^1\setminus\{\infty\}\cong\ABbb^1$. One concludes $V\cong\ABbb^2$. This has the following two implications: \[\pi_1(A)\cong(e),\;\Ocal(A)^*\cong\CBbb^*.\] Thus $H^*(A,\ZBbb)\cong H^*(\PBbb^1,\ZBbb)$ (this isomorphism is abstract, this is not induced by the pull-back map $A\to\PBbb^1$). We claim that $A$ cannot admit a smooth $\ABbb^1$ fibration over $\PBbb^1$. First observe if $A$ admits a smooth $\ABbb^1$ fibration over $\PBbb^1$ then $A$ is isomorphic to the total space of a torsor of a line bundle $L$ over $\PBbb^1$. Moreover this torsor provides a class in $\PBbb(H^1(\PBbb^1,L))\cong\PBbb(\Ext^1(L^{\vee},\Ocal))$. Thus a rank 2 vector bundle $E$ as an extension \[0\to\Ocal\to E\to L^{\vee}\to0.\] Now the projectivisation of $E$ gives us an Hirzebruch surface with a section corresponding to the quotient $E\to L^{\vee}$. It is easy to see that complement of this section is the aforementioned $L$-torsor hence isomorphic to $A$. Let this Hirzebruch surface be $F_n$, for $n\in\ZBbb_{\geq0}$, i.e. $F_n=\PBbb(\Ocal\oplus\Ocal(-n))$. Let the section given by the quotient $\Ocal\oplus\Ocal(-n)\twoheadrightarrow\Ocal(-n)$ is denoted by $C_0$ and the class of a fiber be denoted by $f$. Then $\Pic(F_n)$ is generated by $C_0,f$ and \[C_0^2=-n,\;f^2=0,\;C_0\cdot f=1.\] Now the section $B=F_n\setminus A$ is ample as $A$ is affine. Thus $B\sim C_0+b\cdot f$, where $b>n$. The normal bundle of $B\cong\PBbb^1$ corresponds to the line bundle of degree $B^2=(C_0+bf)^2=2b-n$. Thus the fundamental group of the $S^1$ bundle corresponding to the normal bundle is $\ZBbb/(2b-n)\ZBbb$. This is nothing but the fundamental group at infinity for $A$. Next we compute the canonical bundle $K_A=K_{F_n}|A$. We know that $K_{F_n}\sim -2C_0-(n+2)f$. Also by the localization sequence one gets,
    \[0\to\ZBbb[B]\xrightarrow{1\mapsto(1,b)}\ZBbb[C_0]\oplus\ZBbb[f]\cong\Pic(F_n)\to\Pic(A)\to0.\]
    Thus $\Pic(A)\cong\ZBbb$, with $K_A=(2b-n-2)\cdot l$, where $l$ is a generator of $\Pic(A)$. Now as $b>n$ it is clear that $2b-n-2\geq0$. We note that $(2b-n)-(2b-n-2)=2$. This relation is independent of the chosen Hirzebruch surface and the section of it.  

    Next we will compute these invariants intrinsically from $A$. We follow the same terminology as in example~\ref{example}. Then $A\subset X_3$, $X_3$ is obtained by 3 blow-ups of a Hirzebruch surface $X$. $F$ is the strict transform of the fiber containing the blow-up centers, $\Delta'$ is the strict transform of an ample section $\Delta$ of $X$, $E_1,E_2,E_3$ are the exceptional divisors (or their strict transforms) in $X_3$. $A$ is the complement of the divisor $D'=\Delta'\cup F\cup E_1\cup E_2$. We shall compute the fundamental group of a tubular neighborhood of $D'$ next. We know that \[(\Delta')^2=\alpha,\;F^2=E^2_1=E^2_2=-2.\] Then by \cite[page:11-12]{Mumford} the required fundamental group is generated by the generators $\delta,f,e_1,e_2$ satisfying the following relations:\[\delta^{\alpha}\cdot f=1,\;\delta\cdot e_2\cdot f^{-2}=1,\; f\cdot e_1\cdot e_2^{-2}=1,\;e_2\cdot e_1^{-2}=1.\] Moreover we have the following commutator relations:\[[\delta,f]=[f,e_2]=[e_2,e_1]=1.\] These relations imply, $e_2=e_1^2,f=e_1^3,\delta=e_1^4$ and $e_1^{4\alpha+3}=1$. Thus the required group is cyclic of order $4\alpha+3$. So the fundamental group at infinity of $A$ is $\ZBbb/(4\alpha+3)\ZBbb$. Now we shall compute the $K_A$. Recall the diagram in~\ref{example}
    \[\begin{tikzcd}
	{X_3=\text{Bl}_zX_2} & {X_2=\text{Bl}_yX_1} & {X_1=\text{Bl}_x X} & X
	\arrow["{q_3}", from=1-1, to=1-2]
	\arrow["{q_2}", from=1-2, to=1-3]
	\arrow["{q_1}", from=1-3, to=1-4].
\end{tikzcd}\]
Let the class of exceptional divisor in $X_i$ be denoted by $\varepsilon_i$. Moreover $X$ is a Hirzebruch surface so assume $X\cong F_m$ for $m\geq0$. Thus $K_X\sim-2C_0-(m+2)f$, where as before $C_0$ is the section corresponding to the quotient $\Ocal\oplus\Ocal(-m)\twoheadrightarrow\Ocal(-m)$ and $f$ is the class of the fiber. Now let the chosen ample section $\Delta$ corresponds to the class $C_0+d\cdot f$. As $\Delta^2=\alpha$ so \[(C_0)^2+2d(C_0\cdot f)=2d-m=\alpha\implies d=(\alpha+m)/2.\] As a result we can write $C_0=\Delta-((\alpha+m)/2)\cdot f$. Putting this in the expression of canonical bundle we get:\[K_X\sim-2\Delta+(\alpha-2)\cdot f.\]
Note that the centers of blow-up in each step avoids $\Delta$, hence its strict transform and pull-back in each case are same. So in case of $X_1$ we have \[K_{X_1}\sim-2\Delta+(\alpha-2)\cdot \tilde{f}+(\alpha-2)\cdot \varepsilon_1+\varepsilon_1=-2\Delta+(\alpha-2)\cdot\tilde{f}+(\alpha-1)\cdot \varepsilon_1.\]
Here $\tilde{f}$ denotes the strict transform of $f$. Using similar logic we get that in $X_3$,\[K_{X_3}\sim-2\Delta'+(\alpha-2)F+(\alpha-1)E_1+(2\alpha-2)E_2+(2\alpha-1)E_3.\] It is also clear that $\Delta',F,E_1,E_2,E_3$ forms an integral basis of $\Pic(X_3)$ Thus $\Pic(A)\cong\ZBbb[E_3|_A]$. Thus $K_A=K_{X_3}|_A=(2\alpha-1)E_3|_A$. Clearly $2\alpha-1>0$ and \[(4\alpha+3)-(2\alpha-1)=2\alpha+4\neq2.\] This concludes that $A$ cannot admit a smooth $\ABbb^1$ fibration over $\PBbb^1$.
\end{example}
\begin{remark}\label{fundgpinfty}
    It is worthwhile to note from the above example that: for an affine surface to be a smooth $\ABbb^1$ fiber bundle over $\PBbb^1$ it has to satisfy that the fundamental group at infinity is finite cyclic (say of order $m$). Moreover if the canonical bundle $K_A$ is represented by $n\in \ZBbb\cong\Pic(A)$ then $m-|n|=2$. The above example shows that this condition unavoidable even if one assumes that the cohomology of $A$ is isomorphic to that of $\PBbb^1$ (as M.H.S.) and $\pi_1(A)$ is trivial. In fact as $\pi_1(A)$ is trivial so $\pi_2(A)\cong H_2(A,\ZBbb)$. Also $H_i(A,\ZBbb)\cong0$ for all $i\neq2$. Now choosing a map $S^2\to A$ giving a generator of $\pi_2(A)$ induces isomorphism in integral homology. As both $A$ and $S^2$ are simply connected, homological version of Whitehead's theorem says that this map is a homotopy equivalence. Thus $A$ is homotopic to $S^2$. 
\end{remark}
\section{Appendix}

We give an alternative proof of Theorem~\ref{mainthm} in case of a curve fibration which does not use mixed Hodge modules.

\begin{theorem}
Let $X$ be a smooth surface and
\[
f:X\to C
\]
be a projective morphism with connected fibers to a smooth curve $C$. Suppose that $f$ has exactly one singular fiber, lying over a closed point $c\in C$. If $F$ is a general fiber and $F_c$ is the fiber over $c$, then
\[
\dim H^1(F,\QBbb)\geq\dim H^1(F_c^{\mathrm{red}},\QBbb).
\]

\begin{proof}
Let the singular fiber be denoted by $F_c$. Since the general fiber $F$ is smooth, we have
\[
2g_a(F)=2g(F)=\dim H^1(F,\QBbb).
\]
Moreover, since $f$ is flat,
$
g_a(F)=g_a(F_c),
$
and singular cohomology is insensitive to the non-reduced structure. Hence
\[
H^1(F_c,\QBbb)=H^1(F_c^{\mathrm{red}},\QBbb).
\]
Therefore, it is enough to prove that
\[
2g_a(F_c)\geq \dim H^1(F_c,\QBbb).
\]

First, note that it suffices to prove the result when $F_c^{\mathrm{red}}$ is an $S.N.C.$ divisor in $X$. Indeed, one may take an embedded resolution of singularities of the reduced special fiber to obtain an $S.N.C.$ divisor. During this process, one only blows up reduced closed points lying on $F_c$. Let
\[
q:X'\to X
\]
be the resulting composition of blow-ups. Then
$
Rq_*(\underline{\QBbb}_{X'})
\cong
\underline{\QBbb}_X\oplus L[-2],
$
where $L$ is a skyscraper sheaf supported at finitely many closed points of $F_c^{\mathrm{red}}$. Consequently,
\[
R(f\circ q)_*(\underline{\QBbb}_{X'})
=
Rf_*(Rq_*(\underline{\QBbb}_{X'}))
=
Rf_*(\underline{\QBbb}_X)\oplus L'[-2],
\]
where $L'$ is a skyscraper sheaf supported at $c$. Taking $\Hcal^1$ gives
\[
R^1(f\circ q)_*(\underline{\QBbb}_{X'})
=
R^1f_*(\underline{\QBbb}_X).
\]
Applying the proper base change theorem to both $f\circ q$ and $f$, we conclude that the first cohomology of the reduced special fiber is unchanged after the embedded resolution.

We now prove the result when the singular fiber is reduced and has simple normal crossings. The non-reduced case will be treated later.
Let $N$ denote the number of nodes of $F_c$, and let $r$ denote the number of irreducible components of $F_c$. Since every irreducible component is smooth, we have
\begin{equation}
g_a(F_c)=\sum_i g_a(F_i)+N-r+1.
\end{equation}

We next compute $\dim H^1(F_c,\QBbb)$. Let $F_i^\circ$ denote the complement in $F_i$ of the nodes lying on $F_i$, and let $N_i$ be the number of such nodes. Since every node lies on exactly two irreducible components,
\[
\sum_i N_i=2N,
\qquad
\sum_i1=r.
\]
Moreover,
$
H_c^1(F_i^\circ,\QBbb)
=
\QBbb^{\oplus(N_i-1)}
\oplus
H^1(F_i,\QBbb).
$
We also have the exact sequence
\[
0
\to
\QBbb^{\oplus(N-1)}
\to
\bigoplus_i H_c^1(F_i^\circ,\QBbb)
\to
H^1(F_c,\QBbb)
\to
0.
\]

Considering the dimensions of the terms in the above exact sequence, we obtain
\begin{align*}
    \dim H^1(F_c,\QBbb)
    &=\sum_i(\dim H^1(F_i,\QBbb)+N_i-1)-N+1\\
    &=\sum_i\dim H^1(F_i,\QBbb)+\sum_i N_i-\sum_i1-N+1\\
    &=\sum_i\dim H^1(F_i,\QBbb)+2N-r-N+1\\
    &=\sum_i\dim H^1(F_i,\QBbb)+N-r+1\\
    &=2\cdot\Bigl(\sum_i g_a(F_i)\Bigr)+N-r+1.
\end{align*}
The last equality follows from the fact that each $F_i$ is smooth.

Let $G$ be the dual graph of the nodal curve $F_c$ (cf. definition \ref{dualgraph}). The quantity $N-r+1$ is then equal to $E-V+1$, where $E$ is the number of edges of $G$ and $V$ is the number of vertices. Since $F_c$ is connected, the graph $G$ is also connected.

Observe that if $G$ consists of a single vertex and no edges, then
$
E-V+1=0.
$
Now suppose we add another vertex. To keep the graph connected, we must also add at least one edge. Consequently, the quantity $E-V+1$ remains non-negative. This proves that
\[
2\cdot g_a(F_c)\geq \dim H^1(F_c,\QBbb).
\]

Moreover, observe that if $G$ contains a cycle, then $E-V+1>0$, so the above inequality is strict. Hence equality holds if and only if the dual graph is a tree.

Now we treat the case when $F_c$ is not reduced. Since the problem is local around $c$, we may assume that $C$ is the spectrum of a discrete valuation ring with closed point $c$. We then take a ramified cover $g:(C',c)\to(C,c)$ such that, if $X'$ denotes the normalization of the base change of $X$ over $C'$, then the fiber of $X'$ over $c$ is reduced and has simple normal crossings. The arithmetic genus of the general fiber is unchanged under this procedure. Such a cover exists by the proof of the semistable reduction theorem for families of curves.

The new surface $X'$ has at worst rational singularities (more precisely, singularities of type $A_n$), and hence is Cohen--Macaulay. Therefore, by miracle flatness, the arithmetic genus of the fiber over $c$ is the same as that of the original special fiber. By the previous computation, the desired inequality holds for this new fibration.

Thus it only remains to show that the dimension of the first cohomology of the special fiber after the base change and normalization is greater than or equal to that before the base change and normalization.

\[
\begin{tikzcd}
    {X'} && \\
    & {Y} & S \\
    & {C'} & C
    \arrow["n", from=1-1, to=2-2]
    \arrow["q"', from=1-1, to=3-2]
    \arrow["{g'}", from=2-2, to=2-3]
    \arrow["{f'}", from=2-2, to=3-2]
    \arrow["f", from=2-3, to=3-3]
    \arrow["g"', from=3-2, to=3-3].
\end{tikzcd}
\]

Here $n$ is the normalization map. Let the fibers over $c$ under the maps $f$, $f'$, and $q$ be denoted by $F_c$, $F'_c$, and $F''_c$, respectively. Clearly,
\[
F_c^{\mathrm{red}}=(F'_c)^{\mathrm{red}}.
\]
We now recall the relevant part of the proof of the semistable reduction theorem (cf. \cite{curves}). Let $t$ be a uniformizing parameter of $C$ at $c$. Then, étale locally, a neighbourhood of a point $\alpha\in F_c^{\mathrm{red}}$ is of the following form:

\begin{itemize}
    \item If $\alpha$ is a smooth point of $F_c^{\mathrm{red}}$, then an étale neighbourhood of $\alpha$ is isomorphic to
    \[
    \Spec(R[x,y]/(x^a-t)),
    \]
    where $a$ is the multiplicity of the irreducible component containing $\alpha$.

    \item If $\alpha$ is a node of $F_c^{\mathrm{red}}$, then an étale neighbourhood of $\alpha$ is isomorphic to
    \[
    \Spec(R[x,y]/(x^by^c-t)),
    \]
    where $b$ and $c$ are the multiplicities of the two branches meeting at the node.
\end{itemize}

Here $R=\Ocal_{C,c}$. Let $L$ be the least common multiple of all the multiplicities. We take the ramified cover $(C',c)\to(C,c)$ such that the uniformizing parameter $t$ of $C$ at $c$ maps to $\tau^L$, where $\tau$ is a local uniformizing parameter of $C'$ at $c$. Let $R'=\Ocal_{C',c}$. Then the local étale charts become

\begin{itemize}
    \item At a smooth point of $F_c^{' ,\mathrm{red}}$,
    \[
    \Spec(R'[x,y]/(x^a-\tau^L)).
    \]

    \item At a node of $F_c^{' ,\mathrm{red}}$,
    \[
    \Spec(R'[x,y]/(x^by^c-\tau^L)).
    \]
\end{itemize}

Since we are working over $\CBbb$,
\[
R'[x,y]/(x^a-\tau^L)
=
R'[x,y]\Big/\left(\prod_{i=0}^{a-1}(x-\zeta^i\tau^{L/a})\right),
\]
where $\zeta$ is a primitive $a$-th root of unity. Thus, after normalization, every smooth point of $F_c^{' ,\mathrm{red}}$ has $a$ preimages, each of which is smooth.

Next, consider a node whose two branches have multiplicities $b$ and $c$, and let $\gcd(b,c)=d$. Then
\[
R'[x,y]/(x^by^c-\tau^L)
=
R'[x,y]\Big/\left(
\prod_{i=0}^{d-1}
\left(x^{b/d}y^{c/d}-\gamma^i\tau^{L/d}\right)
\right),
\]
where $\gamma$ is a primitive $d$-th root of unity. Thus the normalization $X'$ factors through $\hat X$ in such a way that the node has $d$ preimages in $\hat X$, each locally defined by the equation
\[
x^{b/d}y^{c/d}-\tau^{L/d}.
\]

If $b\neq c$, then the fiber over $c$ is still non-reduced at these points. Hence we must normalize the ring
\[
R'[x,y]/(x^by^c-\tau^{L'}),
\]
where $\gcd(b,c)=1$ and $L'$ is divisible by $bc$. Choose positive integers $g$ and $h$ satisfying
\[
gb-hc=1
\]
(after interchanging $b$ and $c$ if necessary). Then the normalization is given by
\begin{align*}
R'[x,y]/(x^by^c-\tau^{L'})
&\hookrightarrow
R'[u,v]/(uv-\tau^{L'/bc}),\\
x&\longmapsto u^c,\\
y&\longmapsto v^b,
\end{align*}
where
\[
u=\frac{\tau^{gL'/b}}{x^hy^g},
\qquad
v=\frac{x^hy^g}{\tau^{hL'/c}}.
\]

Now the fiber over $c$ is reduced, and $X'\to\hat X$ has a unique preimage of the node, which is again a node. Moreover, this point is an $A_n$-singularity of $X'$ unless $L'=bc$, in which case it is smooth. Hence the fiber $F_c''$ is a reduced nodal curve. Since the base curve is smooth and the map $q:X'\to C'$ is proper, the fiber $F_c''$ is connected.

From the above discussion, we obtain the following facts concerning the surjective map
\[
F_c''\longrightarrow (F'_c)^{\mathrm{red}}=F_c^{\mathrm{red}}.
\]

\begin{itemize}
    \item[(Fact 1)]\label{F1}
    The preimage of every node is a node, and the preimage of every smooth point is smooth.

    \item[(Fact 2)]\label{F2}
    Let $C_1$ and $C_2$ be two irreducible components of $F_c^{\mathrm{red}}$ meeting at a node $n$. Then every irreducible component $C_1'$ lying over $C_1$ contains a node $n'$ lying over $n$, which is shared with an irreducible component $C_2'$ lying over $C_2$. Consequently, if $N_1$ and $N_2$ denote the numbers of irreducible components lying over $C_1$ and $C_2$, respectively, then the number of preimages of $n$ is at least
    \[
    \max\{N_1,N_2\}.
    \]
\end{itemize}

Let $N$ and $N''$ denote the numbers of nodes of $F_c^{\mathrm{red}}$ and $F_c''$, respectively, and let $r$ and $r''$ denote their numbers of irreducible components. Let $G$ and $G''$ be the corresponding dual graphs. We claim that
\[
N''-r''+1\geq N-r+1.
\]

Henceforth, for any graph $H$, we denote
\[
\#\{\text{edges of }H\}-\#\{\text{vertices of }H\}+1
\]
by $h^1(H)$. In this notation, our goal is to prove
\[
h^1(G'')\geq h^1(G).
\]

We first prove the following reduction. Suppose $\Gfrak$ is a connected subgraph of $G$ (need not be induced), and let $\Gfrak''$ denote its preimage in $G''$. If
\[
h^1(\Gfrak'')\geq h^1(\Gfrak),
\]
then
\[
h^1(G'')\geq h^1(G).
\]

Indeed, if $\Gfrak=G$, there is nothing to prove. Otherwise, choose a vertex
\[
v\in G\setminus\Gfrak.
\]
Suppose $v$ has $n$ preimages in $G''$, and that $v$ is connected to vertices $v_1,\ldots,v_k$ of $\Gfrak$ by $e_1,\ldots,e_k$ edges, respectively. Fix one such vertex $v_i$. By Fact~2, the preimages of $v$ share at least $n e_i$ edges with the preimages of $v_i$.

Let $\hat{\Gfrak}$ be the induced subgraph generated by $\Gfrak$ together with $v$, and let $\hat{\Gfrak}''$ be the induced subgraph generated by $\Gfrak''$ together with all preimages of $v$. Then
\[
h^1(\hat{\Gfrak}'')
\geq
h^1(\Gfrak'')
+n\left(\sum_{i=1}^{k}e_i-1\right)
\geq
h^1(\Gfrak)
+\sum_{i=1}^{k}e_i-1
=
h^1(\hat{\Gfrak}).
\]

Since $G$ is finite, repeating this process finitely many times proves the claim.

It therefore remains to construct a non-empty subgraph $\Gfrak$ satisfying the hypothesis. By construction, there exist at least two vertices $v_1,v_2$ joined by at least one edge. First suppose that the number of edges joining $v_1$ and $v_2$ is $e>1$. Let $\Gfrak$ be the induced subgraph generated by these two vertices. Suppose that $v_1$ and $v_2$ have $n_1$ and $n_2$ preimages, respectively, with $n_1\geq n_2$. By Fact~2, the number of edges between the preimages of $v_1$ and $v_2$ is at least $n_1e$. Hence
\[
h^1(\Gfrak'')
\geq
n_1e-n_1-n_2+1,
\]
while
\[
h^1(\Gfrak)=e-2+1=e-1.
\]
Now we have the following inequalities:
\begin{align*}
    h^1(\Gfrak'')
    &\geq n_1\cdot e-2n_1+1\\
    &=(n_1)\cdot (e-2)+1\\
    &\geq e-2+1=e-1=h^1(\Gfrak).
\end{align*}
The last inequality follows from the assumption that $e\geq2$. Thus $\Gfrak$ satisfies the required property.

Therefore, it remains to consider the case in which there is at most one edge between any two vertices. We first assume that $G$ contains a cycle. Let $\Gfrak$ be such a cycle. Observe that, for a cycle in which every pair of adjacent vertices is connected by a single edge, the quantity $h^1$ is equal to $1$.

Let the vertices of the cycle be $v_1,\ldots,v_k$, where there is an edge between $v_i$ and $v_{i+1}$ for $i=1,\ldots,k$, and we set $v_{k+1}=v_1$. Let $n_i$ denote the number of preimages of $v_i$, and define
\[
n_{i,i+1}=\max\{n_i,n_{i+1}\}.
\]
Then, exactly as before,
\begin{equation}
    h^1(\Gfrak'')
    \geq
    \sum_{i=1}^{k}n_{i,i+1}
    -
    \sum_{i=1}^{k}n_i
    +1
    =
    \sum_{i=1}^{k}(n_{i,i+1}-n_i)+1
    \geq
    1.
\end{equation}
Hence the desired inequality holds in this case as well.

The only remaining case is when $G$ contains no cycles and there is at most one edge between any two vertices. In other words, $G$ is a tree. For a tree, it is clear that
\[
h^1(G)=0.
\]
On the other hand, $G''$ is the dual graph of a connected nodal curve, and by the previous computation,
\[
h^1(G'')\geq0.
\]
Therefore, the desired inequality follows.

We are now ready to prove that
\[
\dim H^1(F_c'',\QBbb)\geq
\dim H^1((F_c')^{\mathrm{red}},\QBbb).
\]
Let $G''$ and $G'$ denote the dual graphs of $F_c''$ and $(F_c')^{\mathrm{red}}$, respectively. Then
\[
\dim H^1(F_c'',\QBbb)
=
\sum_i2g_a(C_i'')+h^1(G''),
\]
and
\[
\dim H^1((F_c')^{\mathrm{red}},\QBbb)
=
\sum_j2g_a(C_j')+h^1(G').
\]

Since the map
\[
F_c''\longrightarrow (F_c')^{\mathrm{red}}
\]
is surjective, every irreducible component of $(F_c')^{\mathrm{red}}$ is dominated by an irreducible component of $F_c''$. Hence
\[
\sum_i2g_a(C_i'')
\geq
\sum_j2g_a(C_j').
\]
Moreover, we proved in the previous paragraph that
\[
h^1(G'')\geq h^1(G').
\]
Combining these two inequalities gives
\[
\dim H^1(F_c'',\QBbb)
\geq
\dim H^1((F_c')^{\mathrm{red}},\QBbb),
\]
which completes the proof.
\end{proof}
\end{theorem}

\bibliographystyle{alpha}
\bibliography{references2}

\end{document}